\documentclass[12pt]{article}
\usepackage{amsfonts,amssymb,amsmath,amsthm}
\usepackage{graphicx}
\usepackage{epstopdf}
\ifpdf
  \DeclareGraphicsExtensions{.eps,.pdf,.png,.jpg}
\else
  \DeclareGraphicsExtensions{.eps}
\fi
\usepackage{pbox}
\usepackage{color,url}
\usepackage{dsfont}
\usepackage{setspace}
\usepackage{graphicx}
\usepackage{epsfig}
\usepackage{float}
\usepackage{stmaryrd}
\usepackage{pifont}
\usepackage{comment}

\newtheorem{thm}{\bf{Theorem}}[section]
\newtheorem{lem}[thm]{\bf{Lemma}}

\newtheorem{df}[thm]{\bf{Definition}}
\newtheorem{cor}[thm]{\bf{Corollary}}

\newtheorem{prop}[thm]{\bf{Proposition}}

\newtheorem{ex}[thm]{\bf{Example}}

\DeclareRobustCommand
  \Compactcdots{\mathinner{\cdotp\mkern-2mu\cdotp\mkern-2mu\cdotp}}

\newcommand{\hilight}[1]{\colorbox{yellow}{#1}}
\newcommand{\dom}{\operatorname{dom}}

\newcommand{\1}{\operatorname{\mathds{1}}}

\newcommand{\Span}{\operatorname{span}}

\newcommand{\Id}{\operatorname{Id}}

\newcommand{\rank}{\operatorname{rank}}

\newcommand{\Proj}{\operatorname{Proj}}

\newcommand{\R}{\operatorname{\mathbb{R}}}
\newcommand{\N}{\operatorname{\mathbb{N}}}
\newcommand{\X}{\operatorname{\mathcal{X}}}
\newcommand{\xm}{\X^{-}}
\newcommand{\y}{\operatorname{\mathcal{Y}}}
\newcommand{\U}{\operatorname{\mathcal{U}}}

\newcommand{\A}{\operatorname{\mathcal{A}}}

\newcommand{\jf}{J_f}
\newcommand{\bbm}{\begin{bmatrix}}
\newcommand{\ebm}{\end{bmatrix}}
\newcommand{\sumdk}{\sum_{i=1}^{k}\left(\prod_{j \neq i}f_j (x^0)\right )\delta^s_{f_i}}
\newcommand{\dfftothei}{\delta^s_{f|f^i}}
\newcommand{\dfminusonefminusi}{\delta^s_{f^{-1}|f^{-i}}}
\newcommand{\doneoverff}{\delta^s_{\frac{1}{f}|f}}

\newcommand{\dfovergg}{\delta^s_{\frac{f}{g}|g}}
\newcommand{\ds}{\delta^s}
\newcommand{\dc}{\delta^c}
\newcommand{\dsf}{\delta^s_f}

\newcommand{\dcf}{\delta^c_f}
\newcommand{\dcg}{\delta^c_g}
\newcommand{\dfxg}{\delta_{f | g}^s}
\newcommand{\nc}{\nabla^c}
\newcommand{\ns}{\nabla^s}
\newcommand{\esfg}{E^s_{fg}}
\newcommand{\esfonefk}{E_{f_1\Compactcdots f_k}^s}
\newcommand{\dsfonefk}{\delta_{f_1\Compactcdots f_k}^s}
\newcommand{\dcfonefk}{\delta_{f_1\Compactcdots f_k}^c}

\newcommand{\esfk}{E_{f^k}^s}
\newcommand{\ecfk}{E_{f^k}^c}
\newcommand{\ecfg}{E_{fg}^c}
\newcommand{\ecfonefk}{E_{f_1\Compactcdots f_k}^c}
\newcommand{\esfoverg}{E_{\frac{f}{g}}^s}
\newcommand{\ecfoverg}{E_{\frac{f}{g}}^c}
\newcommand{\esfminusk}{E_{f^{-k}}^s}
\newcommand{\ecfminusk}{E_{f^{-k}}^c}
\newcommand{\sg}{S_g}
\newcommand{\dsfcircg}{\delta^s_{f \circ g}}
\newcommand{\dcfcircg}{\delta^c_{f \circ g}}
\newcommand{\jsg}{J^s_g}
\newcommand{\jcg}{J^c_g}

\newcommand{\ecfcircg}{E_{f\circ g}^c}
\newcommand{\ncc}{\nabla^{cc}}

\newcommand{\sxtd}{\left ( S^\top \right )^{\dagger}}

\newcommand{\cjt}{\left (J_g^c(\X) \right )^\top}
\newcommand{\cj}{J_g^c(\X)}

\newcommand{\jac}{J_g(x^0)}

\newcommand{\cgfy}{\nc f\left ( g(\X)\right )}
\newcommand{\uu}{\mathbb{U}}
\newcommand{\vv}{\mathbb{V}}

\newcommand{\aff}{\operatorname{aff}}
\newcommand{\gradfvg}{\nabla f_{\vv}\left (g\left (x^0 \right ) \right )}
\newcommand{\jacw}{J_{g_\uu}(x^0)}
\newcommand{\jtw}{\left (J_{g_\uu}(x^0) \right )^\top}

\begin{document}
\title{Error bounds for overdetermined and underdetermined generalized centred simplex gradients}
\author{Warren Hare\thanks{Department of Mathematics, University of British Columbia, Okanagan Campus, Kelowna, B.C. V1V 1V7, Canada. Research partially supported by NSERC of Canada Discovery Grant 2018-03865. warren.hare@ubc.ca, ORCID 0000-0002-4240-3903}\and Gabriel Jarry--Bolduc\thanks{Department of Mathematics, University of British Columbia, Okanagan Campus, Kelowna, B.C. V1V 1V7, Canada. Research partially supported by Natural Sciences and Engineering Research Council (NSERC) of Canada Discovery Grant 2018-03865. gabjarry@alumni.ubc.ca }\and Chayne Planiden\thanks{School of Mathematics and Applied Statistics, University of Wollongong, Wollongong, NSW, 2500, Australia. Research supported by University of Wollongong. chayne@uow.edu.au, ORCID 0000-0002-0412-8445}}
\maketitle\author
\begin{abstract}
 Using the Moore--Penrose pseudoinverse, this work generalizes the gradient approximation technique called centred simplex gradient to allow sample sets containing any number of points. This approximation technique  is called the \emph{generalized centred simplex gradient}.  We develop error bounds and, under a full-rank condition, show that the error bounds have order $O(\Delta^2)$, where $\Delta$ is the radius of the  sample set of points used.  We establish calculus rules for generalized centred simplex gradients, introduce a calculus-based generalized centred simplex gradient and confirm that error bounds for this new approach are also order $O(\Delta^2)$.  We provide several examples to illustrate the results and some benefits of these new methods.
\end{abstract}


\section{Introduction}\label{sec:intro}
Derivative-free optimization (DFO) focuses on the study of optimization algorithms that do not use first-order information within the algorithm. Recent advances in their applications, convergence analysis and practical implementations have fuelled a surge in DFO research (see \cite{Audet2014,audet2017derivative,Conn2009,Custodio2017,HareNutiniTesfamariam2013,larson_menickelly_wild_2019} and citations therein). 

One of the broad classes of DFO algorithms is model-based DFO methods.  These methods rely on accurately approximating first-order information using only function evaluations and then using the approximations within classical optimization algorithms.  For example, using linear interpolation on function values from $n+1$ well-poised sample points in $\R^n$ creates a linear model of the objective function.  The gradient of this linear model, called the {\em simplex gradient}, provides an approximation of the true gradient \cite{Bortz1998,Kelley1999}.

The error bound comparing the simplex gradient and the true gradient dates back to the late 1990s and is known to be order $O(\Delta)$, where $\Delta$ is the radius of the  sample set of evaluated points \cite{Kelley1999}.   This error bound is critical in showing convergence of many first-order model-based methods
\cite[Ch.\ 10 \& 11]{audet2017derivative}.

Simplex gradients, and their associated error bound, are not limited to the setting where exactly $n+1$ interpolation points are used in $\R^n$. In \cite{Conn2008}, the authors study the construction of simplex gradients consisting of $n+1$ interpolation points in $\R^n$, and in \cite{Conn22008}, they extend those results to the cases of fewer (underdetermined models) and more (overdetermined models) than $n+1$ points.  Most notably, they establish error bounds for these cases and find them to be order $O(\Delta)$.  These results were further elaborated in \cite{Regis2015}.

Many other methods of approximating gradients exist \cite{Billups2013,Oeuvray2007,Powell2004,Regis2005,Schonlau1998,Wild2011}. Central to this work is the {\em centred simplex gradient}, which is created by retaining the original points in the sample set and adding their reflection through the reference point (see Definition \ref{def:GCSG}).  This creates an average of two simplex gradients.  Interestingly, the accuracy of the centred simplex gradient is $O(\Delta^2)$ \cite{Kelley1999}.  However, this error bound is only established for the determined case, using exactly $2n$ function evaluations in $\R^n$.  The primary goal of this paper is to establish the error bound for the centred simplex gradient for the underdetermined and overdetermined cases.  This is accomplished using the Moore--Penrose pseudoinverse to define a \emph{generalized centred simplex gradient (GCSG)}, which allows centred simplex gradients to be constructed using a sample set of any finite size.

Returning briefly to our discussion of the simplex gradient, \cite{hare2020,Regis2015} develop calculus rules for the generalized simplex gradient. A secondary goal of this paper is the extension of these results to the GCSG.  This provides the concept of the \emph{generalized centred simplex calculus gradient (GCSCG)}.  We examine this novel gradient approximation and prove that it retains the $O(\Delta^2)$ accuracy of the centred simplex gradient.  Some benefits of the new techniques are illustrated through examples.

The structure of the paper is the following. In Section \ref{sec:prel}, we introduce notation and basic definitions. In Section \ref{sec:errorbounds}, we show that generalized centred simplex gradients inherit the  order of accuracy $O(\Delta^2)$.  We present two error bounds, depending on the number of points  in the sample set. In Section \ref{sec:calc}, we present the calculus rules for the GCSG. In Section \ref{sec:centredsimplexcalculusgradient}, we define the GCSCG, based on the calculus rules from  the previous section. We prove that each of these new techniques has order of accuracy $O(\Delta^2)$.  We provide examples, showing some benefits of the GCSCG compared to the GCSG in certain situations. Section \ref{sec:conclusion} summarizes the work accomplished and suggests some topics to explore in future research.

\section{Preliminaries} \label{sec:prel}

Unless otherwise stated, we use the standard notation found in \cite{rockwets}. The domain of a function $f$ is denoted by $\dom(f)$. The transpose of a matrix $A$ is denoted by $A^\top$. We work in finite-dimensional space $\R^n$ with inner product $x^\top y=\sum_{i=1}^nx_iy_i$ and induced norm $\|x\|=\sqrt{x^\top x}$. We use angle brackets $\langle\cdots\rangle$ to contain an ordered set of vectors. The identity matrix is denoted by $\Id$. We denote by $B(x,\Delta)$ the open ball centred about $x$ with radius $\Delta$. The set of all linear combinations of the vectors in a set $S$ is denoted by $\Span S$. \medskip\\
We next list some definitions and background results.
\begin{df}[Jacobian]
Given a differentiable function $f: \dom(f) \subseteq \R^n\to\R^p$, the \emph{Jacobian} of $f$,
written $\jf$, is the column matrix of all partial derivatives of $f$:
$$\jf=\left[\begin{array}{c c c}\frac{\partial
f}{\partial x_1}&\cdots&\frac{\partial f}{\partial
x_n}\end{array}\right]=\left[\begin{array}{c c c}
\frac{\partial f_1}{\partial x_1}&\cdots&\frac{\partial f_1}{\partial x_n}\\
\vdots&\ddots&\vdots\\
\frac{\partial f_p}{\partial x_1}&\cdots&\frac{\partial f_p}{\partial
x_n}\end{array}\right].$$
\end{df}
\begin{df}[Lipschitz continuity] A function $f: \dom(f) \subseteq \R^n\to\R$ is said to be \emph{Lipschitz continuous} with Lipschitz constant $L\geq0$ if for all $x,y\in\dom (f)$,
$$|f(y)-f(x)|\leq L\|y-x\|.$$If for every $x\in\dom (f) $ there exists a neighbourhood $U$ of $x$ such that $f$ restricted to $U$ is Lipschitz continuous, then $f$ is said to be \emph{locally Lipschitz continuous} on $U$.
\end{df}
We remind the reader that for nonsquare matrices, a generalization of the matrix inverse is the pseudoinverse. The most well-known type of matrix pseudoinverse, which is central to the results of this work, is the Moore--Penrose pseudoinverse.
\begin{df}[Moore--Penrose pseudoinverse] \label{def:mpinverse}
Let $A\in\R^{n\times m}$. The \emph{Moore--Penrose pseudoinverse} of $A$, denoted by $A^\dagger$, is the unique matrix in $\R^{m \times n}$ that satisfies the following four equations:
\begin{align*}
AA^\dagger A&=A,\\A^\dagger AA^\dagger&=A^\dagger,\\(AA^\dagger)^\top&=AA^\dagger,\\(A^\dagger A)^\top&=A^\dagger A.
\end{align*}
\end{df}
The Moore--Penrose inverse $A^\dagger$ is not always an inverse of $A$, but the following two properties hold.
\begin{itemize}
\item If $A$ has full column rank $m$, then $A^\dagger$ is a left-inverse of $A$, that is, $A^\dagger A=\Id_m$.
\item If $A$ has full row rank $n$, then $A^\dagger$ is a right-inverse of $A$, that is, $AA^\dagger=\Id_n$.
\end{itemize}
In order to define the generalized simplex gradient and the GCSG in the sequel, we use the following sets, matrices and terminology.
\begin{df}[Simplex notation]\label{df:simplex}Given $f:\dom(f)  \subseteq \R^n\to\R$ and an ordered sample set of distinct  points
\begin{align*}
\X&=\langle x^0,x^1,\ldots,x^m\rangle=\langle x^0,x^0+d^1,\ldots,x^0+d^m\rangle\subseteq \dom (f),
\intertext{we define}
\X^-&=\langle x^0,x^0-d^1,\ldots,x^0-d^m\rangle\subseteq \dom (f),\\
\Delta&=\max\limits_{i \,\in\{1,\ldots,m\}}\|d^i\| \mbox{\emph{, the radius of }}\X,\\
S=S(\X)&=[d^1~~\cdots~~d^m]\in\R^{n\times m},\\
\ds=\dsf(\X)&=\left[\begin{array}{c}f(x^1)-f(x^0)\\f(x^2)-f(x^0)\\\vdots\\f(x^m)-f(x^0)\end{array}\right]\in\R^m,\\
\dc=\dcf(\X)&=\frac{1}{2}\left[\begin{array}{c}f(x^0+d^1)-f(x^0-d^1)\\f(x^0+d^2)-f(x^0-d^2)\\\vdots\\f(x^0+d^m)-f(x^0-d^m)\end{array}\right]\in\R^m.
\end{align*}
\end{df}
\begin{df}
Let $\X=\langle x^0, x^0+d^1, \dots, x^0+d^m \rangle \subseteq \dom(f)$  be an ordered set of $m+1$ distinct points in $\R^n.$ Then we classify the set $\X$ in exactly one of the following cases:
\begin{itemize}
\item \emph{overdetermined case} if $m>n$ and $\rank S=n$;
\item\emph{determined case} if $m=n$ and $\rank S=n$;
\item\emph{underdetermined case} if $m<n$ and $\rank S=m$;
\item\emph{undetermined case} if S is not full rank.
\end{itemize}
\end{df}

Note that the points in $\X$ are assumed to be distinct for the rest of this paper. 
\begin{df}[Generalized simplex gradient] Let $f:\dom(f) \subseteq \R^n\to\R$ and let $\X$ be an ordered set in $\dom(f).$ The \emph{generalized simplex gradient} (GSG) of $f$ over $\X$ is denoted by $\nabla^sf(\X)$ and defined by
\begin{equation*}\label{eq:simplex}\ns f(\X)=(S^\top)^\dagger\ds.\end{equation*}
\end{df}
\begin{df}[Generalized centred simplex gradient]\label{def:GCSG}
Let $f:\dom(f) \subseteq \R^n\to\R$ and let $\X$ be an ordered set such that $\X \cup \X^{-}$ is in  $\dom(f)$. The \emph{generalized centred simplex gradient} (GCSG) of $f$ over $\X$ is denoted by $\nabla^cf(\X)$ and defined by
\begin{equation*}\label{eq:centred}\nc f(\X)=(S^\top)^\dagger\dc.\end{equation*}
\end{df}
It is easy to show that an equivalent way to compute the GCSG is by using the average of two GSGs:
\begin{equation}\label{eq:altcentred}\nc f(\X)=\frac{1}{2}(\ns f(\X)+\ns f(\X^-)).\end{equation}
Next, we introduce a proposition that clarifies the relation between the GSG and the GCSG and provides a case where both approaches are equal. We require the following two lemmas first.
\begin{lem}
Let $A \in \R^{n \times m}$. Then $(A^\top)^\dagger=(A^\dagger)^\top.$
\end{lem}
\begin{proof}
Since $(A^\dagger)^\top$ and $(A^\top)^\dagger$ satisfy the same four properties in Definition \ref{def:mpinverse}, it follows that they are equal.
\end{proof}
\begin{lem} Let  $A \in \R^{n \times m}$ have full  row rank. Then $\bbm A&-A\ebm^\dagger=\frac{1}{2}\bbm A^\dagger\\-A^\dagger\ebm.$
\end{lem}
\begin{proof}
Since $\bbm A &-A\ebm$ has full row rank, we have
\begin{align*}
    \bbm A&-A\ebm^\dagger&=\bbm A&-A \ebm^\top \left ( \bbm A&-A\ebm \bbm A&-A \ebm^\top \right )^{-1}\\
    &=\bbm A^\top\\-A^\top \ebm \left ( \bbm A&-A \ebm \bbm A^\top\\-A^\top \ebm \right )^{-1}\\
    &=\bbm A^\top\\-A^\top \ebm \left ( 2 \bbm AA^\top \ebm \right )^{-1}\\
    &=\frac{1}{2}\bbm A^\top (AA^\top)^{-1}\\-A^\top (AA^\top)^{-1}   \ebm\\
    &=\frac{1}{2}\bbm A^\dagger\\-A^\dagger \ebm.\qedhere
\end{align*}
\end{proof}
\begin{prop}
Let $f:\dom(f) \subseteq \R^n \to \R$ and let $\X=\langle x^0, x^0+d^1, \dots, x^0+d^m \rangle$ be an ordered set such that $\X \cup \X^{-} \subseteq \dom(f)$ and $S$ has full row rank (determined and overdetermined cases). Let $\mathcal{Y}=\left \langle x^0, x^0+d^1, \dots, x^0+d^m, x^0-d^1, \dots, x^0-d^m \right \rangle$. Then
\begin{align*}
    \ns f(\y)&=\nc f(\X).
\end{align*}
\end{prop}
\begin{proof}
We have
\begin{align*}
    \ns f(\y)&=(S(\y)^\top)^\dagger \dsf (\y)\\
    &=\left (\bbm S(\X) &-S(\X)\ebm^\top\right )^\dagger \bbm \dsf(\X)\\\dsf(\xm)\ebm=\left (\bbm S(\X) &-S(\X)\ebm^\dagger\right )^\top \bbm \dsf(\X)\\\dsf(\xm)\ebm\\
    &=\frac{1}{2}\bbm S(\X)^\dagger\\-S(\X)^\dagger\ebm^\top \bbm \dsf(\X)\\\dsf(\xm)\ebm=\frac{1}{2}\bbm (S(\X)^\top)^\dagger& -(S(\X)^\top)^\dagger  \ebm\bbm \dsf(\X)\\\dsf(\xm)\ebm\\
    &=\frac{1}{2}\left ( (S(\X)^\top)^\dagger \dsf(\X) -(S(\X)^\top)^\dagger \dsf(\xm) \right )\\
    &=\frac{1}{2}\left ( (S(\X)^\top)^\dagger \dsf(\X) +(S(\xm)^\top)^\dagger \dsf(\xm) \right )=\nc f(\X).\qedhere
\end{align*}
\end{proof}
Notice that the GCSG uses an ordered set of points. The reason is that if the position of the reference point $x^0$  is changed in the overdetermined case, then we do not necessarily get the same value. The following example illustrates this situation.
\begin{ex} Consider the sets $\X=\langle -1, 0, 1 \rangle$ and $\X_\alpha=\langle 0, 1, -1 \rangle,$ and the function $f: \R \to \R: y \mapsto y^4.$ Then
\begin{align*}
    \nc f(\X)&=(S(\X)^\top)^\dagger \dcf(\X)\\
    &=\bbm1\\2\ebm^\dagger \frac{1}{2} \bbm 0 -16\\1-81\ebm=-17.6
    \intertext{and}
    \nc f(\X_\alpha)&=\bbm 1 \\-1 \ebm^\dagger \frac{1}{2} \bbm 1 - 1\\ 1 -1 \ebm =0.
\end{align*}
\end{ex}

\section{Error bounds for the GCSG} \label{sec:errorbounds}

This section is dedicated to developing the $O(\Delta^2)$ upper bounds on the error for the GCSG. There are two instances to consider separately; first we look at the determined and overdetermined cases, then the underdetermined case. These two settings have different results, as the number of linearly independent vectors in the simplex differs.

\subsection{Determined and overdetermined cases}

An error bound for the determined case of the GCSG is established in \cite[Theorem 9.13]{ audet2017derivative}, the accuracy of which is measured in terms of $\Delta$. The GCSG error bound in the determined case has order $O(\Delta^2)$. We show that this error bound can be extended to the overdetermined case. To that end, we present Lemma \ref{lem1}, which relies on the multidimensional  second-order Taylor theorem below.

\begin{thm}\emph{\cite[Section 4.3]{Lax2017}}\label{thm:Taylor} Suppose $f:\dom(f)\subseteq \R^n\to\R$ is a $\mathcal{C}^{3}$ function on the open ball $B(x^0,\overline{\Delta}).$  Then for $x^0+d$ in the ball,
\begin{align*}
    f(x^0+d)&=f(x^0)+\nabla f(x^0)^\top d+ \frac{1}{2} d^\top \nabla^2 f(x^0) d+R_2(x^0, d).
\end{align*}
Moreover,
\begin{align*}
    \vert R_2(x^0, d) \vert \leq \frac{L}{6} \Vert d \Vert^3,
\end{align*}
where $L$ is the Lipschitz constant of the Hessian $\nabla^2 f.$
\end{thm}
\begin{lem}\label{lem1}
Let $f:\dom(f) \subseteq \R^n\to\R$ be $\mathcal{C}^{3}$ on $B(x^0,\overline{\Delta})$ and denote by $L$ the Lipschitz constant of $\nabla^2f$. Then for any $d\in B(x^0,\overline{\Delta})$, we have
\begin{equation}\label{eq1}|f(x^0+d)-f(x^0-d)-2\nabla f(x^0)^\top d|\leq\frac{L}{3}\|d\|^3.\end{equation}
\end{lem}
\begin{proof}
From Theorem \ref{thm:Taylor}, we know that
\begin{align}f(x^0+d)&=f(x^0)+\nabla f(x^0)^\top d+\frac{1}{2}d^\top\nabla^2f(x^0)d+R_2(x^0, d),\label{eq2}\\
f(x^0-d)&=f(x^0)-\nabla f(x^0)^\top d+\frac{1}{2}d^\top\nabla^2f(x^0)d+R_2(x^0, -d).\label{eq3}
\end{align}
Subtracting \eqref{eq3} from \eqref{eq2}, we find
\begin{align*}
f(x^0+d)-f(x^0-d)-2\nabla f(x^0)^\top d&=R_2(x^0, d)-R_2(x^0, -d)\\
\Rightarrow |f(x^0+d)-f(x^0-d)-2\nabla f(x^0)^\top d|&=|R_2(x^0, d)-R_2(x^0, -d)|.
\end{align*}
Also from Theorem \ref{thm:Taylor}, we know that
\begin{equation}\label{eq6}|R_2(x^0, \pm d)|\leq\frac{L\|d\|^3}{6}.\end{equation}
Therefore, by \eqref{eq6} and the triangle inequality, we obtain \eqref{eq1}.
\end{proof}
Now we are ready for our first error bound result, for the determined and overdetermined cases.
\begin{thm}\label{thm:det}
Let $f:\dom(f) \subseteq \R^n\to\R$ be $\mathcal{C}^{3}$ on $B(x^0,\overline{\Delta}),$ and denote by $L$ the Lipschitz constant of $\nabla^2f$. Let $\X=\left \langle x^0, x^1, \dots, x^m \right \rangle$ be an ordered set with radius $\Delta<\overline{\Delta}$ and let $\rank S=n$ (determined and overdetermined cases).
 Then
\begin{equation*}\label{eq:determinederror}\|\nc f(\X)-\nabla f(x^0)\|\leq\frac{L\sqrt m}{6}\left\|(\widehat{S}^\top)^\dagger\right\|\Delta^2,\end{equation*}
where $\widehat{S}=S/\Delta$.
\end{thm}
\begin{proof}
 We have
\begin{align*}
    \|\dcf-S^\top\nabla f(x^0)\|&=\sqrt{\sum\limits_{i=1}^m\left(\frac{f(x^0+d^i)-f(x^0-d^i)}{2}-(d^i)^\top\nabla f(x^0)\right)^2}\\
    &\leq \sum\limits_{i=1}^m\left\vert \frac{f(x^0+d^i)-f(x^0-d^i)}{2}-(d^i)^\top\nabla f(x^0)\right \vert.
\end{align*}
Now using Lemma \ref{lem1} and the definition of $\Delta$, we have
\begin{equation*}\label{eq7}\|\dcf-S^\top\nabla f(x^0)\|\leq\frac{L\sqrt m\Delta^3}{6}.\end{equation*}
Since $S^\top$ has full column rank, $(S^\top)^\dagger$ is a left inverse of $S^\top$. Thus,
\begin{align*}
\|\nc f(\X)-\nabla f(x^0)\|&=\left\|(S^\top)^\dagger\dcf-(S^\top)^\dagger S^\top\nabla f(x^0)\right\|\\
&\leq\left\|(S^\top)^\dagger\right\|\|\dcf-S^\top\nabla f(x^0)\|\\
&\leq\left\|(S^\top)^\dagger\right\|\frac{L\sqrt m\Delta^3}{6}\\
&=\left\|(\widehat{S}^\top)^\dagger\right\|\frac{L\sqrt m\Delta^2}{6}.\qedhere
\end{align*}
\end{proof}

\subsection{Underdetermined case}

When developing the underdetermined case, we obtain the error bound by considering $f$ not on all of $\R^n$, but on a subspace that is dependent on the linearly independent vectors in $S$. Suppose $\uu=\Span S\neq\R^n$ (i.e.\ $\rank S<n$). Note that $\uu+x^0$ is the affine hull of $\X$, which we denote by $\aff\X$. Since all of our sample points lie in $\uu+x^0$, it is unreasonable to expect the ability to estimate gradients accurately outside of this affine hull. The following example demonstrates this problem.
\begin{ex}\label{ex:proj}
Let $f:\R^3\to\R: y\mapsto ay_1+(a+1)y_2+(4-a)y_3$, $a\in\R$. Consider the sample set$$\X=\left\langle x^0,x^1,x^2\right\rangle=\left\langle\left[\begin{array}{c}0\\0\\0\end{array}\right],\left[\begin{array}{c}1\\0\\1\end{array}\right],\left[\begin{array}{c}0\\1\\1\end{array}\right]\right\rangle.$$Then regardless of the value of $a$, we have
$$f(x^0)=0,\qquad f(x^1)=a+4-a=4,\qquad f(x^2)=a+1+4-a=5.$$
Therefore, it is impossible to determine $\nabla f$ using $\X$.
\end{ex}
We wish to create an approximate gradient that is accurate on the subspace $\uu$. Note that $\uu=\{y :y_3=y_1+y_2\}$. Define $\Proj_{\uu}$ as the projection operator onto $\uu\subseteq\R^n$. We restrict $f$ to the domain $\uu$ by defining the following function:
$$f_{\uu}(y)=f(x^0+\Proj_{\uu}(y-x^0)).$$Then by \cite[Lemma 4.2]{hare2019chain}, we have the following useful relationship between the projection of the gradient of $f$ onto $\uu$ and the so-called $\U$-gradient of $f$ denoted by $\nabla f_{\uu}$:
$$\nabla f_{\uu}=\Proj_{\uu}\nabla f.$$
The function $f_{\uu}$ effectively restricts the domain of $f$ to the subspace $\uu$ and provides us with gradient information on $\aff\X$.
 
\begin{ex} Let $f:\R^3\to\R:y \mapsto a^\top y=a_1y_1+a_2y_2+a_3y_3$ and$$\X=\left\langle x^0,x^1,x^2\right\rangle=\left\langle\left[\begin{array}{c}0\\0\\0\end{array}\right],
\left[\begin{array}{c}1\\0\\1\end{array}\right],\left[\begin{array}{c}0\\1\\1\end{array}\right]\right\rangle.$$We have $$S=\left[\begin{array}{c c}1&0\\0&1\\1&1\end{array}\right].$$

By \emph{\cite[\S 2.2 eq. (3)]{vusmoothness} (see \eqref{eq:projsum2} below)}, the projection of $y$ onto $\uu$ is given by
$$\Proj_{\uu}y=\frac{1}{3} \bbm 2&-1&1\\-1&2&1\\1&1&2\ebm \bbm y_1\\y_2\\y_3 \ebm=\frac{1}{3}\left[\begin{array}{c}2y_1-y_2+y_3\\-y_1+2y_2+y_3\\y_1+y_2+2y_3\end{array}\right].$$
Thus,
\begin{align*}f_{\uu}(y)=f(x^0+\Proj_{\uu}(y))&=a_1\frac{2y_1-y_2+y_3}{3}+a_2\frac{-y_1+2y_2+y_3}{3}+a_3\frac{y_1+y_2+2y_3}{3}\\
&=\frac{2a_1-a_2+a_3}{3}y_1+\frac{-a_1+2a_2+a_3}{3}y_2+\frac{a_1+a_2+2a_3}{3}y_3\end{align*}
and $$\nabla f_{\uu}(y)=\frac{1}{3}\left[\begin{array}{c}2a_1-a_2+a_3\\-a_1+2 a_2+a_3\\a_1+a_2+2a_3\end{array}\right].$$
On the other hand, we have
\begin{align*}
\Proj_{\uu}\nabla f(y)=\Proj_{\uu}a=\frac{1}{3} \bbm 2&-1&1\\-1&2&1\\1&1&2\ebm \left[\begin{array}{c}a_1\\a_2\\a_3\end{array}\right]=\frac{1}{2}\left[\begin{array}{c}2a_1-a_2+a_3\\-a_1+2 a_2+a_3\\a_1+a_2+2a_3\end{array}\right].
\end{align*}
\end{ex}
In the underdetermined case below, we will work with the Moore--Penrose inverse $S^\dagger$.  By \cite[\S 2.2 eq. (3)]{vusmoothness}, if $S$ has full column rank, then
\begin{equation}\label{eq:projsum2}
\Proj_{\uu}y=S(S^\top S)^{-1}S^\top y=(S^\top)^\dagger S^\top y.
\end{equation}
With this in mind, we are ready to provide an error bound for the underdetermined case.
\begin{thm} \label{thm:under}
Let $f:\dom(f) \subseteq \R^n\to\R$ be $\mathcal{C}^{3}$ on $B(x^0,\overline{\Delta})$ and denote by $L$ the Lipschitz constant of $\nabla^2f$.
Let $\X=\langle x^1,\ldots,x^m\rangle$ be an ordered set   with radius $\Delta<\overline{\Delta}$. Let $\rank S=m<n$ (underdetermined case) and let $\uu=\Span S.$  Define $\widehat{S}=S/\Delta$. Then
$$
\|\nabla^c f (\X) - \nabla f_{\uu}(x^0)\| =
\|\Proj_{\uu}\nabla^c f(\X)-\Proj_{\uu}\nabla f(x^0)\|
\leq \frac{L\sqrt m}{6}\left\|(\widehat{S}^\top)^\dagger\right\|\Delta^2.$$
\end{thm}
\begin{proof}  By definition, $\nabla f_{\uu}(x^0)=\Proj_{\uu}\nabla f(x^0)$. Also, note that for all $i \in \{0, 1, \dots, m\}$, we have $f(x^i)=f_{\uu}(x^i)$. So $\nc f(\X)=\nc f_{\uu} (\X),$ and by definition, $\nc f_{\uu}(\X)= \Proj_{\uu} \nabla^c f(\X).$ Thus, the first equality holds.

From the definition of Lipschitz continuity, restricting $f$ to $\uu=\Span S\subseteq\R^n$ does not alter $\nabla^2f$ from being $L$-Lipschitz. The projection of a point onto $\uu$ is given by \eqref{eq:projsum2}, which yields
\begin{align*}
\|\Proj_{\uu}\nabla^cf(\X)-\Proj_{\uu}\nabla f(x^0)\|&=\|(S^\top)^\dagger S^\top(S^\top)^\dagger\delta^c_f-(S^\top)^\dagger S^\top\nabla f(x^0)\|.
\end{align*}
Since $S$ has full column rank, $S^\top$ has full row rank, which yields $S^\top (S^\top)^\dagger =\Id_m$. Hence,
\begin{align*}
\|\Proj_{\uu}\nabla^cf(\X)-\Proj_{\uu}\nabla f(x^0)\|&=\|(S^\top)^\dagger\delta^c_f-(S^\top)^\dagger  S^\top\nabla f(x^0)\|\\
&\leq\|(S^\top)^\dagger\|\|\delta^c_f-S^\top\nabla f(x^0)\|.
\end{align*}
Using Lemma \ref{lem1}, we have
\begin{align*}
\|\Proj_{\uu}\nabla^cf(\X)-\Proj_{\uu}\nabla f(x^0)\|&\leq \|(S^\top)^\dagger\|\frac{L\sqrt m}{6}\Delta^3\\
&\leq\|(\widehat{S}^\top)^\dagger\|\frac{L\sqrt m}{6}\Delta^2.\qedhere
\end{align*}
\end{proof}
We observe that this result is almost identical to the Theorem \ref{thm:det} result, except that it is in the reduced space rather than in $\R^n$. With only $m<n$ linearly independent vectors in $S$, this is the best result possible for the underdetermined case.

\section{Calculus rules}\label{sec:calc}

In this section, we provide calculus rules for the GCSG. Throughout this section, all  functions have a domain contained in $\R^n$  and map to $\R$. The calculus formulae for the GCSG follow directly from the calculus rules for the GSG  presented in \cite{hare2020} by using \eqref{eq:altcentred}.  Before introducing the calculus rules for the GSCG, let us recall the definition of the \emph{product difference vector} and the calculus rules for the GSG (Table  \ref{tab:calcrulesgsg}) introduced in \cite{hare2020}.

\begin{df}[Product difference vector]
Let $\X=\langle x^0, x^0+d^1, \dots, x^0+d^m \rangle$ be an ordered set of $m+1$ distinct points contained in $\dom(f)$ and $\dom(g)$. The \emph{product difference vector} of $f$ and $g$  over $\X$ is denoted by $\delta^s_{f|g}(\X)$ and defined by
\[
\dfxg (\X)=\begin{bmatrix}(f(x^0+d^1)-f(x^0))(g(x^0+d^1)-g(x^0))\\(f(x^0+d^2)-f(x^0))(g(x^0+d^2)-g(x^0)) \\\vdots\\(f(x^0+d^m)-f(x^0))(g(x^0+d^m)-g(x^0))\end{bmatrix}.
\]
\end{df}
\begin{table}[ht]
\caption{\textbf{Calculus rules for the GSG}}\label{tab:calcrulesgsg}
\center{
\begin{tabular}{ |p{1.5cm}||p{6.8cm}|p{7.1cm}| }

 \hline
\textbf{Rule}& \textbf{Formula} &\textbf{$E^s$}\\
 \hline
 Product of $2$ & $f(x^0)\ns g(\X)+g(x^0)\ns f(\X)+E^s_{fg}$&$\left (S^\top\right )^\dagger \dfxg$\\
 Product of $k$&$\sum_{i=1}^{k}\left(\prod_{j\neq i}f_j (x^0) \right)\ns f_i(\X)+ E^s_{f_1\cdots f_n}$ &$\sxtd \left ( \delta^s_{f_1\cdots f_n} -
 \sumdk \right )$\\
 Positive power &$n[f(x^0)]^{n-1}\ns f(\X)+E^s_{f^n}$ & $\sxtd\left ( \sum_{i=1}^{n-1}[f(x^0)]^{n-1-i}\dfftothei \right )$\\
 Negative power &$-n[f(x^0)]^{-n-1} \ns f(\X)-E_{f^{-n}}$&$\frac{\sxtd}{[f(x^0)]^n}\left ( n \doneoverff - \sum_{i=1}^{n-1}[f(x^0)]^{1+i} \dfminusonefminusi \right )$\\
 Quotient&$\frac{g(x^0)\ns f(\X)-f(x^0)\ns g(\X)}{[g(x^0)]^2} -E_{\frac{f}{g}}$ &$\frac{\sxtd}{g(x^0)} \dfovergg$ \\
 \hline
\end{tabular}}
\end{table}

Note that the GCSG is a linear operator. Indeed, this follows from the facts that the GSG is a linear operator \cite[Proposition 9] {Regis2015} and the GCSG is the average of two GSGs.  Using this, we can adjust the product rule for the GSG to a product rule for the GCSG.

\begin{prop}[GCSG product rule]\label{prop:productrule} Let $\X=\langle x^0,x^0+d^1,\ldots,x^0+d^m\rangle$ be an ordered set of $m+1$ points such that $\X \cup \X^{-}$ is in $\dom(f)$ and $\dom(g)$. Then
\begin{equation*}\nc(fg)(\X)=f(x^0)\nc g(\X)+g(x^0)\nc f(\X)+\ecfg(\X),\end{equation*}
where $\ecfg(\X)=\frac{1}{2}\left(\esfg(\X)+\esfg (\X^-)\right)$.
\end{prop}
\begin{proof}
We have
\footnotesize\begin{align*}
\nc (fg)(\X)&=\frac{1}{2}(\ns(fg)(\X)+\ns (fg)(\X^-))\\
&=\frac{1}{2}(f(x^0)\ns g(\X)+g(x^0)\ns f(\X)+\esfg(\X)+f(x^0)\ns g(\X^-)+g(x^0)\ns f(\X^-)+\esfg(\X^-))\\
&=f(x^0)\frac{1}{2}(\ns g(\X)+\ns g(\X^-))+g(x^0)\frac{1}{2}(\ns f(\X)+\ns f(\X^-))+\frac{1}{2}(\esfg (\X)+\esfg (\X^-))\\
&=f(x^0)\nc g(\X)+g(x^0)\nc f(\X)+\ecfg(\X).\qedhere
\end{align*}
\normalsize
\end{proof}

The averaging technique used in Proposition \ref{prop:productrule} can be used to create calculus rules for the product of $k$ functions, positive powers, and negative powers.  We omit proofs for the next three results, as they are straightforward.

\begin{cor}[GCSG product rule, $k$ functions]
Let $f_i:\dom(f_i) \subseteq \R^n\to\R$ for all $i\in\{1,2,\ldots,k\}$, $k\geq 2$ and let $\X=\langle x^0,x^0+d^1,\ldots,x^0+d^m\rangle$ be an ordered set of $m+1$ points such that $\X \cup \X^{-}$ is in $\dom(f_i)$ for all $i$. Then
\begin{equation*}\nc(f_1\cdots f_k)(\X)=\sum\limits_{i=1}^k\prod\limits_{j\neq i}f_j(x^0)\nc f_i(\X)+\ecfonefk(\X),\end{equation*}
where
\begin{equation*}\ecfonefk (\X)=\frac{1}{2}\left(\esfonefk (\X)+\esfonefk (\X^-)\right).\end{equation*}
\end{cor}
\begin{cor}[GCSG power rule]
Let $\X=\langle x^0,x^0+d^1,\ldots,x^0+d^m\rangle$ be an ordered set of $m+1$ points such that $\X \cup \X^{-}$ is in $\dom(f)$. Let  $k\in\N$. Then
\begin{equation*}\nc f^k(\X)=k(f(x^0))^{k-1}\nc f(\X)+\ecfk(\X),\end{equation*}
where
\begin{equation*}\ecfk(\X)=\frac{1}{2}\big(\esfk(\X)+\esfk(\X^-)\big).\end{equation*}
\end{cor}
\begin{prop}[GCSG quotient rule] \label{prop:gcsgquotient}
Let $\X=\langle x^0,x^0+d^1,\ldots,x^0+d^m\rangle$ be an ordered set of $m+1$ points such that $\X \cup \X^-$ is in $\dom(f) \cap \dom(g)$ and  for which $g(x^0)$, $g(x^0\pm d^1)$, $\ldots$, $g(x^0\pm d^m)$ are all nonzero. Then
\begin{equation*}\nc\left(\frac{f}{g}\right)(\X)=\frac{g(x^0)\nc f(\X)-f(x^0)\nc g(\X)}{g^2(x^0)}-\ecfoverg (\X),\end{equation*}
where
\begin{equation*}\ecfoverg (\X)=\frac{1}{2}\Big(\esfoverg (\X)+\esfoverg (\X^-)\Big).\end{equation*}
\end{prop}
\begin{cor}[GCSG power rule, negative exponent] \label{cor:gcsgpowerneg}
Let $\X=\langle x^0,x^0+d^1,\ldots,x^0+d^m\rangle$ be an ordered set of $m+1$ points such that $\X \cup \X^-$ is in $\dom(f)$   and for which $f(x^0)$, $f(x^0\pm d^1), \ldots, f(x^0\pm d^m)$ are all nonzero. Let $k\in\N$. Then
\begin{equation*}\nc f^{-k}(\X)=-k f^{-k-1}(x^0)\nc f(\X)-\ecfminusk (\X),\end{equation*}
where
\begin{equation*}\ecfminusk (\X)=\frac{1}{2}\left(\esfminusk(\X)+\esfminusk (\X^-)\right).\end{equation*}
\end{cor}
\noindent Our final calculus rule for the GCSG is the chain rule. For this, we require some additional notation. Let $f:\dom(f) \subseteq \R^p\to\R$ and $g: \dom(g) \subseteq \R^n\to\R^p$, where
\begin{equation*}g(y)=\left[\begin{array}{c}g_1(y)\\g_2(y)\\\vdots\\g_p(y)\end{array}\right]\in\R^p.\end{equation*}
Let $\X=\langle x^0,x^0+d^1,\ldots,x^0+d^m\rangle$ be an ordered set of $m+1$ distinct points such that $\X \cup \X^-$ is in $\dom(g)$ and define
\begin{align*}
g(\X)&=\langle g(x^0),g(x^0+d^1),\ldots,g(x^0+d^m)\rangle\\
&=\langle g(x^0),g(x^0)+h^1,\ldots,g(x^0)+h^m\rangle,\\
g(\X)^{-}&= \langle g(x^0), g(x^0)-h^1, \ldots, g(x^0)-h^m \rangle
\end{align*}
where $h^i=g(x^0+d^i)-g(x^0)$ for all $i\in\{1,2,\ldots,m\}$, to be ordered sets of $m+1$ points such that $g(\X) \cup g(\X)^-$ is in $\dom(f)$. Denote
\begin{align*}
\sg=S(g(\X))&=[g(x^0+d^1)-g(x^0)~~\cdots~~g(x^0+d^m)-g(x^0)]\\
&=[h^1~~\cdots~~h^m]\in\R^{p\times m}.
\end{align*}
Now we introduce the generalized centred simplex Jacobian of $g$ over $\X$.
\begin{df}[Generalized centred simplex Jacobian] Define the function $g: \dom(g) \subseteq \R^n\to\R^p$, $y\mapsto[g_1(y)~~g_2(y)~~\cdots~~g_p(y)]^\top$. Let $\X$ be an ordered set of $m+1$ points such that $\X \cup \X^-$ is in $\dom(g)$. Then the \emph{generalized centred simplex Jacobian} of $g$ over $\X$, denoted by $\jcg(\X)$, is the $p\times n$ real matrix defined by
 \begin{equation*}
\jcg (\X)=\left[\begin{array}{c}\nc g_1(\X)^\top\\\nc g_2(\X)^\top\\\vdots\\\nc g_p(\X)^\top\end{array}\right].
\end{equation*}
\end{df}
\noindent With these terms defined, we are ready to present the GCSG chain rule. Note that
\begin{align*}
\dcf (g(\X))&=\frac{1}{2}\left[\begin{array}{c}f(g(x^0)+h^1)-f(g(x^0)-h^1)\\\vdots\\f(g(x^0)+h^m)-f(g(x^0)-h^m)\end{array}\right].\\
\neq \dcfcircg (\X)&=\frac{1}{2}\left[\begin{array}{c}f(g(x^0+d^1))-f(g(x^0-d^1))\\\vdots\\f(g(x^0+d^m))-f(g(x^0-d^m))\end{array}\right].
\end{align*}
For this reason, the chain rule for the GCSG cannot be obtain using a similar approach to the one for the GSG \cite[Theorem 15]{hare2020}.
\begin{prop}[GCSG chain rule]
Let the functions  $f:\dom(f) \subseteq \R^p\to\R$, $g:\dom(g) \subseteq \R^n\to\R^p$ and let $\X=\langle x^0,x^0+d^1,\ldots,x^0+d^m\rangle$ be an ordered set of $m+1$ points  such that $\X \cup \X^-$ is in $\dom(g)$ and $g(\X) \cup g(\X)^-$ is in $\dom(f)$. Then
\begin{equation*}
\nc (f\circ g)(\X)=\jcg (\X)^\top\nc f(g(\X))-E,
\end{equation*}
where
\begin{align*}
E&=(S^\top)^\dagger\dcg(\X)(\sg^\top)^\dagger\widetilde{E}+(S^\top)^\dagger\widehat{E}\dcf(g(\X))-(S^\top)^\dagger\widetilde{E}\widehat{E},\\
\widetilde{E}&=\dcf(g(\X))-\dcfcircg(\X),\\
\widehat{E}&=\dcg(\X)(\sg^\top)^\dagger-\Id.
\end{align*}
\end{prop}
\begin{proof}
We have
\begin{align*}
\nc (f\circ g)(\X)&=(S^\top)^\dagger\dcfcircg(\X)\\
&=(S^\top)^\dagger\big(\dcg(\X)(\sg^\top)^\dagger-\widehat{E}\big)\dcfcircg(\X).
\end{align*}Thus,
\begin{equation*}
\nc(f\circ g)(\X)=(S^\top)^\dagger\big(\dcg(\X)(\sg^\top)^\dagger-\widehat{E}\big)\big(\dcf(g(\X))-\widetilde{E}\big).
\end{equation*}
Expanding the left-hand side, we obtain
\small\begin{align*}
\nc (f\circ g)(\X)&=(S^\top)^\dagger\left[\dcg(\X)(\sg^\top)^\dagger\dcf (g(\X))-\dcg(\X)(\sg^\top)^\dagger\widetilde{E}-\widehat{E}\dcf(g(\X))+\widetilde{E}\widehat{E}\right]\\
&=\jcg(\X)^\top\nc f(g(\X))-(S^\top)^\dagger\dcg(\X)(\sg^\top)^\dagger\widetilde{E}-(S^\top)^\dagger\widehat{E}\dcf(g(\X))+(S^\top)^\dagger\widetilde{E}\widehat{E}\\
&=\jcg(\X)^\top\nc f(g(\X))-E.\qedhere
\end{align*}\normalsize
\end{proof}

\section{The generalized centred simplex calculus gradient} \label{sec:centredsimplexcalculusgradient}

In the previous section, we developed calculus rules for the GCSG that involve basic calculus rules plus a term $E$.  In this section, we see that we can eliminate the $E$ terms in all the calculus rules to create new gradient approximation techniques.  The error bounds for these new techniques remain $O(\Delta^2)$. We name these techniques the \emph{generalized centred simplex calculus gradient (GCSCG)}.

Table \ref{tab1} below summarizes the calculus results of Section \ref{sec:calc}.
\begin{table}[H]
\caption{Calculus rules for the GCSG}\label{tab1}
\begin{tabular}{|l||l|l|l|}\hline
Rule&Formula&$E^c$&$E^s$\\\hline
\scriptsize\pbox{2cm}{Product\\$fg$}&\scriptsize $f(x^0)\nc g+g(x^0)\nc f+\ecfg$&\scriptsize $\frac{1}{2}\left(\esfg(\X)+\esfg(\X^-)\right)$&\scriptsize $(S^\top)^\dagger\delta^s_{f|g}$\\\hline
\scriptsize\pbox{2cm}{Product\\ $f_1\Compactcdots f_k$}&\scriptsize $\sum\limits_{i=1}^k\prod\limits_{j\neq i}f_j(x^0)\nc f_i+\ecfonefk$&\scriptsize $\frac{1}{2}\left(\esfonefk(\X)+\esfonefk(\X^-)\right)$&\scriptsize $(S^\top)^\dagger\left(\dsfonefk-\sum\limits_{i=1}^k\prod\limits_{j\neq i}f_j(x^0)\dsf \right)$\\\hline
\scriptsize Power&\scriptsize $kf^{k-1}(x^0)\nc f+\ecfk$&\scriptsize $\frac{1}{2}\left(\esfk(\X)+\esfk (\X^-)\right)$&\scriptsize$(S^\top)^\dagger\left(\sum\limits_{i-1}^{k-1}f^{k-1-i}(x^0)\delta_{f|f^{-i}}^s\right)$ \\\hline
\scriptsize Quotient&\scriptsize $\frac{g(x^0)\nc f-f(x^0)\nc g}{g^2(x^0)}-\ecfoverg$&\scriptsize $\frac{1}{2}\left(\esfoverg(\X)+\esfoverg (\X^-)\right)$&\scriptsize$\frac{(S^\top)^\dagger}{g(x^0)}\delta_{\frac{f}{g}|g}^s$ \\\hline
\scriptsize\pbox{2cm}{Negative\\power}&\scriptsize$-kf^{-k-1}(x^0)\nc f-\ecfminusk$&\scriptsize$\frac{1}{2}\left(\esfminusk(\X)+\esfminusk (\X^-)\right)$&\scriptsize$\frac{(S^\top)^\dagger}{f^k(x^0)}\left(k\delta_{\frac{1}{f}|f}^s-\sum\limits_{i=1}^{k-1}f^{1+i}(x^0)\delta_{f^{-1}|f^{-i}}^s\right)$\\\hline
\scriptsize Chain&\scriptsize $(\jcg)^\top\nc f(g(\X))-\ecfcircg$&\scriptsize\pbox{4cm}{$(S^\top)^\dagger\dcg (\X)(\sg^\top)^\dagger\widetilde{E}+$\\$(S^\top)^\dagger\widehat{E}\dcf (g(\X))-(S^\top)^\dagger\widetilde{E}\widehat{E}$} &\scriptsize$(S^\top)^\dagger\left(\sg^\top(\sg^\top)^\dagger-\Id\right)\dsf (g(\X))$ \\\hline
\end{tabular}
\end{table}
\noindent We introduce the notation $\ncc$ to represent the GCSCG. In the sequel, we formalize the formulae.
\begin{df}[GCSCG]Let $f:\dom(f) \subseteq \R^n\to\R, g: \dom(g)\subseteq \R^n \to \R$ and let $\X=\langle x^0,x^0+d^1,\ldots,x^0+d^m\rangle$ be an ordered set such that $\X \cup \X^-$ is in $\dom(f) \cap \dom(g)$.
The GCSCG of $fg$ over $\X$ is
\begin{equation}\label{eq:cc1}\ncc (fg)(\X)=f(x^0)\nc g(\X)+g(x^0)\nc f(\X).\end{equation}
Let $f_i:\R^n\to\R$ for all $i\in\{1,\ldots,k\}$, $k\geq 2,$ and let $\X$ be an ordered set  such that $\X \cup \X^-$ is in $\dom(f_i)$ for all $i$. The GCSCG of $f_1\cdots f_k$ over $\X$ is
\begin{equation}\label{eq:cc2}\ncc (f_1\cdots f_k)(\X)=\sum\limits_{i=1}^k\prod\limits_{j\neq i}f_j(x^0)\nc f_i(\X).\end{equation}
The GCSCG of $f^k$ over $\X$ is
\begin{equation}\label{eq:cc3}\ncc f^k(\X)=kf^{k-1}(x^0)\nc f(\X),\end{equation}where $f(x^0)$ is nonzero whenever $k-1<0$.\medskip\\
Let $g(x^0)\neq0$. The GCSCG of $\frac{f}{g}$ over $\X$ is
\begin{equation}\label{eq:cc4}\ncc \left(\frac{f}{g}\right)(\X)=\frac{g(x^0)\nc f(\X)-f(x^0)\nc  g(\X)}{g^2(x^0)}.\end{equation}
Let $a\in (0, \infty)$. The GCSCG of $a^f$ over $\X$ is
\begin{equation}\label{eq:cc6}\ncc a^{f(\X)}=a^{f(x^0)}\nc f(\X)\ln a.\end{equation}
Let $f(x^0)\neq0$ and $a\in (0, \infty)$. The GCSCG of $\log_af$ over $\X$ is
\begin{equation}\label{eq:cc7}\ncc \log_af(\X)=\frac{1}{f(x^0)\ln a}\nc f(\X).\end{equation}
Let $f:\dom(f) \subseteq \R^p\to\R$ and $g:\dom(g)\subseteq\R^n\to\R^p$. Let $\X$ be an ordered set such that $\X \cup \X^-$ is in $\dom(g)$ and $g(\X) \cup g(\X)^-$ is in $\dom(f).$ The GCSCG of $f\circ g$ over $\X$ is
\begin{equation}\label{eq:cc5}\ncc (f\circ g)(\X)=\jcg (\X)^\top\nc f(g(\X)).\end{equation}
\end{df}
We point out that the GCSCG is less restrictive than the GCSG. For instance, \eqref{eq:cc3} only requires $f(x^0)$ to be nonzero when $k-1<0$, which is not sufficient in Corollary \ref{cor:gcsgpowerneg}. The quotient rule presented in \eqref{eq:cc4} requires only $g(x^0)$ to be nonzero, whereas Proposition \ref{prop:gcsgquotient} requires all of $g(x^0), g(x^0\pm d^1), \dots, g(x^0\pm d^m)$ to be nonzero. Lastly, the GCSG $\nc\ln f(\X)$ requires $f$ to be positive for all points in $\X$ and $\X^-$, so that $\delta_{\ln f}^c(\X)$ is well-defined. However, only $f(x^0)$ must be nonzero in \eqref{eq:cc7} and $f$ is not restricted at any other point in $\X$ or $\X^-$.\par The preceding seven equations of approximate gradients are summarized in Table \ref{tab2} below for quick reference.
\begin{table}[H]\centering
\caption{Calculus rules for the GCSCG}\label{tab2}
\begin{tabular}{|l||l l|}\hline
Rule&Formula $\ncc$&\\\hline
Product $fg$&$f(x^0)\nc g(\X)+g(x^0)\nc f(\X)$&\eqref{eq:cc1}\\\hline
Product $f_1\Compactcdots f_k$&$\sum\limits_{i=1}^k\prod\limits_{j\neq i}f_j(x^0)\nc f_i(\X)$&\eqref{eq:cc2}\\\hline
Power&$kf^{k-1}(x^0)\nc f(\X)$&\eqref{eq:cc3}\\\hline
Quotient&$\frac{g(x^0)\nc f(\X)-f(x^0)\nc g(\X)}{g^2(x^0)}$&\eqref{eq:cc4}\\\hline
Exponential&$a^{f(x^0)}\nc f(\X)\ln a$&\eqref{eq:cc6}\\\hline
Logarithmic&$\frac{1}{f(x^0)\ln a}\nc f(\X)$&\eqref{eq:cc7}\\\hline
Chain&$\jcg (\X)^\top\nc f(g(\X))$&\eqref{eq:cc5}\\\hline
\end{tabular}
\end{table}
\noindent The next step is to show that GCSCG has controlled error.

\subsection{Error bounds for the GCSCG}

In this section, we demonstrate that the GCSCG is a valid approximation method, in the sense that we can define an error bound between the approximations and the true values of the gradients at $x^0$. Furthermore, we show that the error bounds are all $O(\Delta^2)$. We provide some examples along the way, to show the accuracy gain that can be made by using the GCSCG. The four propositions below follow from applying Theorems \ref{thm:det} and \ref{thm:under} to the appropriate results from Section \ref{sec:calc}. We provide the proof for Proposition \ref{prop:nccfgerrorbound} as a demonstration and omit proofs for the other three results. In the following propositions, we  use $\uu=\Span S \subseteq \R^n.$ Note that if $S$ has full row rank, then $\uu=\R^n$ and $f_\uu=f.$

\begin{prop}[$\ncc (fg)$ error bound] \label{prop:nccfgerrorbound}
 Let $f:\dom(f) \subseteq \R^n\to\R, g:\dom(g) \subseteq \R^n \to \R$ be $\mathcal{C}^{3}$ on $B(x^0, \overline{\Delta})$ and denote by $L_f$ and $L_g$ the Lipschitz constants  of $\nabla^2 f$ and $\nabla^2 g$.
Let $\X=\langle x^0, x^1, \ldots, x^m \rangle$ be an ordered set  with radius $\Delta<\overline{\Delta}$ such that $S$ has full rank. Let $\uu=\Span S$. Then
\begin{equation*}
\|\ncc (fg)(\X)-\nabla(fg)_{\uu}(x^0)\|\leq\frac{\sqrt m}{6}\big(L_g|f(x^0)|+L_f|g(x^0)|\big)\big\|(\widehat{S}^\top)^\dagger\big\|\Delta^2.
\end{equation*}
\end{prop}
\begin{proof}
Note  that $f(x^i)=f_{\uu}(x^i)$ and $g(x^i)=g_{\uu}(x^i)$ for all $i \in \{0, 1, \ldots, m\}.$ We have
\begin{align*}
&\|\ncc(fg)(\X)-\nabla(fg)_{\uu}(x^0)\|\\
=&\|f(x^0)\nc g(\X)+g(x^0)\nc f(\X)-f_{\uu}(x^0)\nabla g_{\uu}(x^0)-g_{\uu}(x^0)\nabla f_{\uu}(x^0)\|\\
\leq&|f(x^0)|\|\nc g(\X)-\nabla g_{\uu}(x^0)\|+|g(x^0)|\|\nc f(\X)-\nabla f_\uu(x^0)\|\\
\leq&\frac{\sqrt m}{6}\left(L_g|f(x^0)|+L_f|g(x^0)|\right)\big\|(\widehat{S}^\top)^\dagger\big\|\Delta^2
\end{align*}
by Theorem \ref{thm:det} or Theorem \ref{thm:under} as appropriate.
\end{proof}
\begin{prop}[$\ncc (f_1\Compactcdots f_k)$ error bound]\label{prop:ccproduct}
Let $f_i:\dom(f_i) \subseteq \R^n\to\R$ be $\mathcal{C}^{3}$ on $B(x^0, \overline{\Delta})$ and denote by $L_i$ the Lipschitz constants  of $\nabla^2 f_i$ for each $i$.
Let $\X=\langle x^0, x^1, \ldots, x^m \rangle$ be an ordered set with radius $\Delta<\overline{\Delta}$ such that $S$ has full rank. Let $\uu= \Span S.$  Then
\begin{equation*}\|\ncc (f_1\Compactcdots f_k)(\X)-\nabla(f_1\Compactcdots f_k)_\uu(x^0)\|\leq\frac{\sqrt m}{6}\sum\limits_{i=1}^k\prod\limits_{j\neq i}|f_j(x^0)|L_i\big\|(\widehat{S}^\top)^\dagger\big\|\Delta^2.\end{equation*}
\end{prop}
\noindent In some situations, the error bound above is zero. Corollary \ref{cor:ccproduct} below gives sufficient conditions for the GCSCG $\nabla_{cc}(f_1\Compactcdots f_k)$ to be perfectly accurate.
\begin{cor}\label{cor:ccproduct}Let the assumptions of Proposition \ref{prop:ccproduct} hold. If any of the following holds:
\begin{itemize}
\item[(a)]$f_i(x^0)=f_j(x^0)=0$ for some $i,j\in\{1,\ldots,k\}, i\neq j$;
\item[(b)]$f_i$ is a polynomial of order less than three for all $i\in\{1,\ldots,k\}$;
\item[(c)]$f_i$ is a polynomial of order less than three and $f_i(x^0)=0$ for some $i\in\{1,\ldots,k\}$,
\end{itemize}then
\begin{equation*}
\ncc (f_1\Compactcdots f_k)(\X)=\nabla(f_1\Compactcdots f_k)_\uu(x^0).
\end{equation*}
\end{cor}
\noindent When $S$ has full  rank,  Corollary \ref{cor:ccproduct} tells us that if just one of the $k$ functions is linear or quadratic and is equal to zero at $x^0$, then $\ncc (f_1\Compactcdots f_k)(\X)$ is equal to $\nabla (f_1\Compactcdots f_k)_\uu(x^0),$ regardless of the nature of the other $k-1$ functions. The same result is obtained if just two of the $k$ functions are equal to zero at $x^0$, no matter the form of the other functions.

\begin{prop}[$\ncc f^k$ error bound]\label{prop:power}
Let $f:\dom(f) \subseteq \R^n\to\R$ be $\mathcal{C}^3$  on $B(x^0, \overline{\Delta})$ and denote by $L$ the  Lipschitz constant of $\nabla^2 f$.  Let $f(x^0)\neq0$ whenever $k-1<0$.
Let $\X=\langle x^0, x^1, \ldots, x^m \rangle$ be an ordered set with radius $\Delta<\overline{\Delta}$ such that $S$ has full rank. Let $\uu=\Span S$ and $k \in \R$. Then
\begin{equation*}
\|\ncc (f^k)(\X)-\nabla (f^k)_\uu(x^0)\|\leq\frac{L\sqrt m}{6}|k||f(x^0)|^{k-1}\big\|(\widehat{S}^\top)^\dagger\big\|\Delta^2.
\end{equation*}
\end{prop}
\begin{prop}[$\ncc \left(\frac{f}{g}\right)$ error bound]\label{prop:quotient} 
Let $f:\dom(f) \subseteq \R^n\to\R, g:\dom(g)\subseteq \R^n \to \R$ be $\mathcal{C}^{3}$ on $B(x^0,\overline{\Delta})$  and denote by  $L_f, L_g$ the Lipschitz constants of $\nabla^2 f, \nabla^2 g$ respectively.
Let $\X=\langle x^0, x^1, \ldots, x^m \rangle$ be an ordered set with radius $\Delta<\overline{\Delta}$ such that $S$ has full rank. Let $\uu=\Span S$.  Assume $g(x^0)\neq0$. Then
\begin{equation*}
\left\|\ncc \left(\frac{f}{g}\right)(\X)-\nabla\left(\frac{f}{g}\right)_\uu(x^0)\right\|\leq\frac{\sqrt m}{6}\left(L_f\left|\frac{1}{g(x^0)}\right|+L_g\left|\frac{f(x^0)}{g^2(x^0)}\right|\right)\big\|(\widehat{S}^\top)^\dagger\big\|\Delta^2.
\end{equation*}
\end{prop}
The error bound involving the chain rule ($\ncc (f\circ g)$) requires new techniques, so we include the proof of Proposition \ref{prop:cccomp}.
\begin{prop}[error bound for $\ncc (f\circ g)$]\label{prop:cccomp}
Let $g:\dom(g) \subseteq \R^n \to \R^p$,  $f: \dom(f) \subseteq \R^p \to \R$  be $\mathcal{C}^{3}$ on  $B(g(x^0),\overline{\Delta}_g)$ and $B(x^0,\overline{\Delta})$ respectively  and denote by $L_{\nabla^2 f} , L_{\nabla^2 g}$ the Lipschitz constants of  $\nabla^2 f$ and $\nabla^2 g$. Denote by $L_{g_i}$ the Lipschitz constant of $g_i$ on  $B(x^0,\overline{\Delta})$ for each  $i \in \{1,2, \dots, p\}.$ Let $\X=\langle x^0, x^1, \ldots, x^m \rangle$ be an ordered set with radius $\Delta<\overline{\Delta}$ and let $g(\X)=\langle g(x^0), g(x^1), \ldots, g(x^m) \rangle$ be an ordered set with radius $\Delta_g<\overline{\Delta}_g$. Assume that $S$ and $\sg$ have full rank. Let $\uu=\Span S \subseteq \R^n$ and $\vv= \Span \sg \subseteq \R^p.$  Then
\footnotesize\begin{equation*}
\|\ncc (f\circ g)(\X)-\nabla(f\circ g)_\uu (x^0)\|\leq \frac{\sqrt{m}\;p}{6}\big(\sqrt{m} \; L_{g_*} \;L_{\nabla^2 f}\big\|(\widehat{S}_g^\top)^\dagger\big\|+ \Vert \nabla f(g(x^0)) \Vert  L_{\nabla^2 g_*} \big)\big\|(\widehat{S}^\top)^\dagger\big\|\Delta^2_{*},\\
\end{equation*}\normalsize
where
\begin{align*}
\Delta_{*}&= \max\left \{ \Delta, \Delta_g \right\},\\
L_{g_{*}}&= \max \{L_{ g_i}: i=1, \dots, p\},\\
 L_{ \nabla^2 g_{*}}&= \max \{L_{\nabla^2 g_i}: i=1, \dots, p\}.
 \end{align*}
\end{prop}
\begin{proof}
We have
\begin{align*}
\Vert  \ncc (f \circ g)(\X)-\nabla(f\circ g)_\uu(x^0)\Vert&=\Vert \cjt \nc f(g(\X)) - \jtw \nabla f_\vv \left ( g_\uu \left ( x^0\right ) \right )\Vert.
\end{align*}
Note that $g_\uu(x^0)=g(x^0+\Proj_\uu(x^0-x^0))=g(x^0).$ We obtain
\begin{align*}
&\Vert  \ncc (f \circ g)(\X)-\nabla(f\circ g)_\uu(x^0)\Vert\\
=&\Vert \cjt \nc f(g(\X)) - \jtw \nabla f_\vv \left ( g \left ( x^0\right ) \right )\Vert\\
=&\Vert \cjt \cgfy - \cjt \gradfvg\\
\quad&+\cjt \gradfvg - \jtw \nabla f_{\vv} \left ( g\left ( x^0\right ) \right )\Vert\\
\leq&\Vert \cjt \Vert \Vert \cgfy -\gradfvg \Vert + \Vert \gradfvg \Vert \Vert \left ( \cj- \jacw \right )^\top \Vert.
\end{align*}
Let us find a bound for $\Vert \cgfy-\gradfvg \Vert.$ If $\sg$ has full column rank, then
\begin{align*}
\Vert \cgfy-\gradfvg \Vert &\leq \frac{\sqrt{m}}{6} L_{\nabla^2 f}\big\|(\widehat{S}_g^\top)^\dagger\big\|\Delta^2_g
\end{align*}
by Theorem \ref{thm:under}. If $\sg$ has full row rank, then
\begin{align*}
    \gradfvg&=\Proj_\vv \nabla f (g(x^0))\\
    &=(\sg^\top)^\dagger \sg^\top \nabla f (g(x^0))\\
    &=\nabla f (g(x^0)).
\end{align*}
Hence
\begin{align*}
    \Vert \cgfy-\gradfvg \Vert &= \Vert \cgfy -\nabla f (g(x^0)) \Vert\\
    &\leq \frac{\sqrt{m}}{6} L_{\nabla^2 f}\big\|(\widehat{S}_g^\top)^\dagger\big\|\Delta^2_g
\end{align*}
by Theorem \ref{thm:det}. Next we find a bound for $\Vert \left (\cj-\jacw \right )^\top \Vert.$ We have
\begin{align*}
\Vert \left (\cj-\jacw \right )^\top \Vert &= \left \Vert \begin{bmatrix}\left (\nc g_1(\X)-\nabla (g_1)_{\uu}(x^0) \right )^\top\\ \vdots \\ \left ( \nc g_p(\X) -\nabla (g_p)_{\uu} (x^0) \right )^\top \end{bmatrix}^\top \right \Vert\\
&\leq \Vert \nc g_1(\X)-\nabla (g_1)_{\uu}(x^0) \Vert + \dots + \Vert \nc g_p(\X)-\nabla (g_p)_{\uu} (x^0) \Vert.\\
\end{align*}
If $S$ has full row rank, then $\nabla (g_i){_\uu}(x^0)=\nabla (g_i)(x^0)$ for all $i \in \{1, 2, \dots, p\}.$ By Theorem \ref{thm:det}, we obtain
\begin{align}
    &\Vert \nc g_1(\X)-\nabla (g_1)_{\uu}(x^0) \Vert + \dots + \Vert \nc g_p(\X)-\nabla (g_p)_{\uu} (x^0) \Vert \notag\\
    &\leq  \frac{\sqrt{m}}{6} \Delta\big\|(\widehat{S}^\top)^\dagger\big\|\left ( L_{\nabla^2 g_1} +\dots +L_{\nabla^2 g_p} \right ) \notag\\
    &\leq \frac{\sqrt{m}\;p}{6} L_{ \nabla^2 g_*}\big\|(\widehat{S}^\top)^\dagger\big\|\Delta^2. \label{eq:bound}
\end{align}
If $S$ has full column rank, then \eqref{eq:bound} is obtained by Theorem \ref{thm:under}. Finally, let us find a bound for  $\Vert \cjt \Vert$. We have
\begin{align*}
    \Vert \cjt \Vert&=\left \Vert \begin{bmatrix}\nc g_1(\X)^\top\\ \vdots \\ \nc g_p(\X)^\top \end{bmatrix}^\top \right \Vert\\
&\leq \Vert \nc g_1(\X) \Vert + \dots + \Vert \nc g_p(\X)\Vert\\
&\leq \big\|(\widehat{S}^\top)^\dagger\big\|\left \Vert \frac{\delta^c_{g_1}(\X)}{\Delta} \right \Vert+ \dots +  \big\|(\widehat{S}^\top)^\dagger\big\|\left \Vert \frac{\delta^c_{g_p}(\X)}{\Delta} \right \Vert\\
&\leq\big\|(\widehat{S}^\top)^\dagger\big\|\sqrt{m} \; L_{g_1}+ \dots +\big\|(\widehat{S}^\top)^\dagger\big\|\sqrt{m} \; L_{g_p}\\
&\leq \sqrt{m}\;p \; L_{g_*}\big\|(\widehat{S}^\top)^\dagger\big\|.
\end{align*}
All together,
\begin{align*}
&\Vert \nc (f \circ g)(\X)-\nabla(f\circ g)_\uu(x^0)\Vert\\
\leq&\sqrt{m}\;p \; L_{ g_*}\big\|(\widehat{S}^\top)^\dagger\big\|\frac{\sqrt{m}}{6} L_{\nabla^2 f}\big\|(\widehat{S}_g^\top)^\dagger\big\|\Delta_g^2+ \Vert \nabla f_{\vv}(g(x^0)) \Vert \frac{\sqrt{m}\;p}{6} L_{ \nabla^2 g_*} \big\|(\widehat{S}^\top)^\dagger\big\|\Delta^2\\
\leq&\frac{\sqrt{m}\;p}{6}\big(\sqrt{m} \; L_{g_*} \;L_{\nabla^2 f}\big\|(\widehat{S}_g^\top)^\dagger\big\|+ \Vert \nabla f_{\vv}(g(x^0)) \Vert  L_{\nabla^2 g_*} \big)\big\|(\widehat{S}^\top)^\dagger\big\|\Delta^2_{*}.
\end{align*}
Since $\Vert \nabla f_{\vv}(g(x^0)) \Vert \leq \Vert \nabla f (g(x^0)) \Vert,$ we obtain the final result. \qedhere
\end{proof}
\noindent Analysing the error bound of Proposition \ref{prop:cccomp}, we find two cases where it is zero. Corollary \ref{cor:cccomp} presents these cases.
\begin{cor}\label{cor:cccomp}
Let the assumptions of Proposition \ref{prop:cccomp} hold. If either of the following holds:
\begin{itemize}
\item[(a)]$g$ is a constant function;
\item[(b)]$f$ and $g$ are polynomials of order less than three,
\end{itemize}then
\begin{equation*}
\ncc (f\circ g)(\X)=\nabla(f\circ g)_\uu(x^0).
\end{equation*}
\end{cor}
\begin{ex}
Consider $f:\R\to\R:y\mapsto y^2$, $g:\R\to\R:y\mapsto y^2+1$ and $\X=\langle2,3\rangle$. We compute the absolute error for $\ncc (f \circ g)(\X)$ and the value of $\nc (f \circ g)(\X).$ Note that $S$ and $\sg$ have full row rank. Hence, $\nabla (f \circ g)_\uu (x^0)=\nabla (f \circ g)(x^0).$ We have
\begin{align*}
\ncc (f\circ g)(\X)&=\jcg (\X)^\top\nc f(g(\X))\\
&=(S^\top)^\dagger\dcg (\X)(\sg^\top)^\dagger\dcf (g(\X))\\
&=1\cdot\frac{1}{2}(10-2)\frac{1}{5}\cdot\frac{1}{2}(100-0)\\
&=40
\end{align*}and the true derivative $\nabla (f \circ g)(x^0)$ is
\begin{equation*}
\frac{d}{dy}(y^2+1)^2\Big\vert_{y=2}=40.
\end{equation*}
Therefore, the absolute error $\| \ncc (f \circ g)(\X)-\nabla (f \circ g)(\X)\|=0.$ The error bound in Proposition \ref{prop:cccomp} is also equal to zero since $f$ and $g$ are quadratic functions. The GCSG $\nc(f\circ g)(\X)$ does not return the exact value of the derivative. Indeed, we have
\begin{align*}
\nc(f\circ g)(\X)&=(S^\top)^\dagger\dcfcircg(\X)\\
&=1\cdot\frac{1}{2}(9+1)^2-(1+1)^2)=48.
\end{align*}
\end{ex}

The next example demonstrates Corollary \ref{cor:cccomp}.

\begin{ex}
Consider $f:\R^3 \to \R:y \mapsto \alpha(y_1^2+y_2^2+y_3^2),$ $\alpha \in \R,$  $g:\R^3 \to \R^2:y \mapsto \bbm y_2-2y_1&y_1+y_2&y_1y_2+y_2\ebm^\top$ and the sample set $\X=\left \langle \bbm 1\\2\ebm, \bbm 2\\2 \ebm, \bbm 1\\3 \ebm\right \rangle.$ We compute the absolute error for $\ncc (f \circ g)(\X).$ Note that $g(\X)=\left \langle \bbm 0\\3\\4\ebm, \bbm -2\\4\\6\ebm, \bbm 1\\4\\6\ebm\right \rangle$ and
\begin{align*}
    S&=\bbm 1&0\\0&1 \ebm, \quad \sg=\bbm -2&1\\1&1\\2&2\ebm.
\end{align*}
We see that $S$ has full row rank and $\sg$ has full column rank, so $\nabla (f \circ g)_\uu(x^0)=\nabla (f \circ g)(x^0)$. We obtain
\begin{align*}
    \ncc (f \circ g)(\X)&=(\jcg(\X))^\top \nc f(g(\X))\\
    &=\bbm -2&1&2\\1&1&2\ebm\bbm 0\\ 4.4\alpha \\8.8\alpha \ebm\\
    &=\alpha\bbm 22\\22\ebm,
\end{align*}
and
\begin{align*}
    \nabla (f \circ g)(x^0)&=(\jac)^\top \nabla f(g(x^0))\\
    &=\bbm -2&1&2\\1&1&2\ebm \bbm 0\\ 6\alpha \\ 8\alpha \ebm\\
    &=\alpha \bbm 22\\ 22\ebm.
\end{align*}
Therefore, the absolute error is
\begin{align*}
\|\ncc (f\circ g)(\X)-\nabla(f\circ g)_\uu(x^0)\|&=\|\ncc (f\circ g)(\X)-\nabla(f\circ g)(x^0)\|=0.
\end{align*}
The error bound in Proposition \ref{prop:cccomp} is also equal to zero since $f$ and $g$ are quadratic functions.
\end{ex}
The next two propositions are proved using the same technique as Proposition \ref{prop:nccfgerrorbound}.
\begin{prop}[$\ncc a^f$ error bound]\label{prop:ccexp}
Let $f:\dom(f)\subseteq \R^n\to\R$ be $\mathcal{C}^{3}$ on $B(x^0,\overline{\Delta})$ and denote by $L$ the Lipschitz constant of $\nabla^2 f.$
Let $\X=\langle x^0, x^1, \ldots, x^m \rangle$ be an ordered set  with radius  $\Delta<\overline{\Delta}$ such that $S$ has full rank. Let $\uu=\Span S$ and $a>0$. Then
\begin{equation*}
\big\|\ncc a^{f(\X)}-\nabla a^{f_\uu(x^0)}\big\|\leq\big|a^{f(x^0)}\ln a\big|\frac{L\sqrt m}{6}\big\|(\widehat{S}^\top)^\dagger\big\|\Delta^2.
\end{equation*}
\end{prop}
\noindent Examining the error bounds of Proposition \ref{prop:ccexp}, we see that it is zero whenever $L=0$. This is the case when $f$ is a polynomial of order less than three. The following examples illustrate this for both cases.
\begin{ex}
Consider $f:\R^2\to\R:y\mapsto y_1^2+y_2^2$ and $\X=\left\langle\left[\begin{array}{c}1\\1\end{array}\right],\left[\begin{array}{c}2\\1\end{array}\right],\left[\begin{array}{c}1\\2\end{array}\right]\right\rangle$. We compute the absolute error for $\ncc e^{f(\X)}$ and the value of $\nc e^f(\X).$ Note that $S$ has full row rank, hence, $\nabla e^{f_\uu(x^0)}=\nabla e^{f(x^0)}.$ We have
\begin{align*}
\ncc e^{f(\X)}&=e^{f(x^0)}\nc f(\X)\\
&=e^2(S(\X)^\top)^\dagger\delta_f^c(\X)\\
&=e^2\left[\begin{array}{c c}1&0\\0&1\end{array}\right]\frac{1}{2}\left[\begin{array}{c}5-1\\5-1\end{array}\right]\\
&=\left[\begin{array}{c}2e^2\\2e^2\end{array}\right]\approx\left[\begin{array}{c}14.78\\14.78\end{array}\right],
\end{align*}
and
\begin{align*}
\nabla e^{f(x^0)}&=\nabla e^{y_1^2+y_2^2}=[2e^2~~2e^2]^\top.
\end{align*}
So the absolute error is equal to zero. The error bound in Proposition \ref{prop:ccexp} is also equal to zero since $f$ is a quadratic function. Also,
\begin{align*}
\nc e^{f(\X)}&=\left(\widehat{S}^\top\right)^\dagger\delta_{e^f}^c(\X)\\
&=\left[\begin{array}{c c}1&0\\0&1\end{array}\right]\frac{1}{2}\left[\begin{array}{c}e^5-e^1\\e^5-e^1\end{array}\right]\approx\left[\begin{array}{c}72.85\\72.85\end{array}\right].
\end{align*}
\end{ex}
\begin{ex}
Consider $f:\R^2\to\R:y\mapsto y_1^2+y_2^2$ and $\X=\left\langle\left[\begin{array}{c}1\\1\end{array}\right],\left[\begin{array}{c}2\\1\end{array}\right] \right\rangle$. We compute the absolute error for $\ncc e^{f(\X)}.$ Note that $S$ has full column rank. We obtain
\begin{align*}
\ncc e^{f(\X)}&=e^{f(x^0)}\nc f(\X)\\
&=\begin{bmatrix}2e^2\\0\end{bmatrix}
\end{align*}
Also,
\begin{align*}
    \nabla e^{f_\uu (x_0)}&=e^{f_\uu(x^0)}\nabla f_\uu(\X)\\
    &=e^{f(x^0)} \Proj_\uu \nabla f(x_0)\\
    &=e^2 \bbm 1&0\ebm^\dagger \bbm1&0 \ebm \bbm 2\\2\ebm=\bbm 2e^2\\0\ebm,
\end{align*}
so the absolute error is
\begin{equation*}
\left\|\ncc e^{f(\X)}-\nabla e^{f_\uu(x^0)}\right\|=0.
\end{equation*}
The error bound in Proposition \ref{prop:ccexp} is also zero since $f$ is a quadratic function.
\end{ex}
\begin{prop}[$\ncc \log_af(\X)$ error bound]\label{prop:cclog}
 Let $f:\dom(f) \subseteq \R^n\to\R$ be $\mathcal{C}^{3}$ on $B(x^0, \overline{\Delta})$  with $f(x^0)\neq0$, and denote by $L$ the Lipschitz constant  of $\nabla^2 f$ on $B(x^0, \overline{\Delta}).$ Let $\X=\langle x^0, x^1, \ldots, x^m \rangle $ be an ordered set with radius $\Delta<\overline{\Delta}$ such that $S$ has full rank. Let $\uu=\Span S$ and $a>0$. Then
\begin{equation*}
\|\ncc \log_af(\X)-\nabla\log_af_\uu(x^0)\| \leq \left|\frac{1}{f(x^0)\ln a}\right|\frac{L\sqrt m}{6}\big\|(\widehat{S}^\top)^\dagger\big\|\Delta^2.
\end{equation*}
\end{prop}
\noindent Examining the error bounds of Proposition \ref{prop:cclog}, we see that it is zero whenever $f$ is a polynomial of order less than three. The example below illustrates this situation.
\begin{ex}
Consider $f:\R^2\to\R:y\mapsto y_1^2+2y_2^2-3$ and $\X=\left\langle\left[\begin{array}{c}2\\2\end{array}\right],\left[\begin{array}{c}3\\2\end{array}\right],\left[\begin{array}{c}2\\3\end{array}\right]\right\rangle$. We compute the error bound for $\ncc \ln f(\X)$ and $\nc \ln f(\X).$  Note that $S$ has full row rank, so $\nabla \ln f_\uu (x^0)=\nabla \ln f(x^0).$ We have
\begin{align*}
\ncc \ln f(\X)&=\frac{1}{9}\left[\begin{array}{c c}1&0\\0&1\end{array}\right]\frac{1}{2}\left[\begin{array}{c}14-6\\19-3\end{array}\right]\\
&=\left[\begin{array}{c}\frac{4}{9}\\\frac{8}{9}\end{array}\right],
\end{align*}and the true gradient is
\begin{equation*}
\nabla\ln f(x^0)=\left[\begin{array}{c}\frac{1}{9}\cdot4\\\frac{1}{9}\cdot8\end{array}\right]=\left[\begin{array}{c}\frac{4}{9}\\\frac{8}{9}\end{array}\right].
\end{equation*}
Therefore, the absolute error is equal to zero. The error bound in Proposition \ref{prop:cclog} is also zero since $f$ is a quadratic function. The GCSG is
\begin{align*}
\nc \ln f(\X)&=\left[\begin{array}{c c}1&0\\0&1\end{array}\right]\frac{1}{2}\left[\begin{array}{c}\ln14-\ln6\\\ln19-\ln3\end{array}\right]\\
&\approx\left[\begin{array}{c}0.4236\\0.9229\end{array}\right].
\end{align*}
\end{ex}

\section{Conclusion} \label{sec:conclusion}
Generalized centred simplex gradients provide a formula to approximate gradients regardless of the number of points in the sample set. In the underdetermined case, an error bound with order $O(\Delta^2)$ is defined by restricting the function to a subspace of $\R^n.$ In the overdetermined case, the error bound remains order $O(\Delta^2)$. Thereafter, we showed that calculus rules for generalized centred simplex gradients can be  written in a way similar to those for the true gradients plus a term $E.$ Removing the term $E$ from the calculus rules leads to new approaches to approximate gradients that also have error bounds of order $O(\Delta^2)$. If the true objective functions are linear or quadratic, then the new approaches result in perfect accuracy.  Corollaries \ref{cor:ccproduct} and \ref{cor:cccomp} provide several cases where the product rule and the chain rule are perfectly accurate.

Recent work has been done in the reduction of calculation time and storage space needed to use simplex gradients.  In \cite{Coope2019Efficient}, Coope and Tappenden start with the knowledge that a simplex gradient in $\R^n$ can require $O(n^3)$ operations and $O(n^2)$ storage units, and then reduces both of them to $O(n)$ under reasonable conditions.  A valuable next step would be to confirm if these techniques also work for generalized centred simplex gradients. Another future research direction would be to investigate error bounds when the matrix $S$ does not have full rank (the undetermined case).

\def\cprime{$'$}

\end{document}